\documentclass[10pt]{amsart}

\usepackage{amsmath, amssymb, amsthm, mathrsfs, float
}

\newtheorem{thm}{Theorem}[section]
\newtheorem{lem}[thm]{Lemma}

\newtheorem{cor}[thm]{Corollary}

\newtheorem{prop}[thm]{Proposition}
\newtheorem{rem}[thm]{Remark}

\newtheorem{ques}{Question}
\newtheorem*{hypo}{Hypothesis}

\newcommand{\N}{\ensuremath{\mathbb{N}}}
\newcommand{\A}{\ensuremath{\mathrm{A}}}
\newcommand{\Oh}{\ensuremath{\mathrm{O}}}
\newcommand{\oh}{\ensuremath{\mathrm{o}}}
\newcommand{\nP}{\ensuremath{\mathcal{P}}}
\renewcommand{\l}{\ensuremath{\ell}}

\renewcommand{\l}{\ell}

\keywords{Euler's function, totients, multiplicity, distribution}

\subjclass[2010]{11A25 \and 11N64}

\begin{document}

\title{$2$-adic stratification of totients}

\author{Andr\'{e} Contiero}
\address{Departamento de Matem\'atica, ICEx, UFMG Av. Ant\^onio Carlos 6627, 30123-970 Belo Horizonte MG, Brazil}
\email{contiero@ufmg.br
}
\author{Davi Lima}
\address{Instituto de Matem\'atica, UFAL. Av. Lourival de Melo Mota, s/n, 57072-900 Macei\'o AL, Brazil}
\email{davimat@impa.br}

\begin{abstract}
In this paper we study the multiplicities and the asymptotic behaviour of the numbers of totients in the strata given by $2$-adic valuation.
\end{abstract}

\maketitle

\section{Introduction}

One of the main multiplicative functions in number theory is the well known Euler's 
totient function, whose value for a positive integer number $n$ is 
$$\phi(n):=\#\{1\leq a\leq n\,\, \vert\, \gcd(a,n)=1\}.$$ 
A particular subject of study is the set $\mathscr{V}$ of \textit{totients}, i.e. the 
set of the values taken by Euler's function. $$\mathscr{V}:=\{1,2,4,6,8,10,12, 16, 18,\dots\}\,.$$ 
The distribution of suitable totients is yet subject of investigation of many authors 
and from many perspectives. 
One of the main objects of study is the \textit{multiplicity} of a given 
positive integer $m\in\mathbb{N}$, that is the number of 
elements in its pre-image 
under $\phi$, namely $$\A(m):=\vert\phi^{-1}(m)\vert .$$ 
Equivalently, $\mathrm{A}(m)$ is 
the number of solutions of the equation $\phi(x)=m$. 

%

Since Euler's $\phi$-function is multiplicative, and for every prime 
number $p$ one has $\phi(p^n)=p^{n-1}(p-1), \,\forall\,n\geq 1$, we 
easily see that totients bigger than $1$ are even numbers. Hence the 
most naive stratification of $\mathscr{V}$ can be made
by taking totients with a fixed $2$-adic valuation. 
More precisely, for each $\l\geq 1$ we pick up the following subset of $\mathscr{V}$,
$$\mathscr{V}^{\l}:=\{m\in \mathscr{V}\,;\,m \equiv 2^{\l} \mod 2^{\l+1} \}.$$
As usual, for each real number $x\in\mathbb{R}$ we write 
$\mathscr{V}^{\l}(x):=\{m\in \mathscr{V}^{\l} ;m\le x\}$, analogously for $\mathscr{V}(x)$. 
We also write $\mathrm{V}^{\l}(x)$ and $\mathrm{V}(x)$ for the number of elements in 
$\mathscr{V}^{\l}(x)$ and in $\mathscr{V}(x)$, respectively.
For a fixed positive integer $\l>1$ two very natural questions can be made.

\begin{ques}\label{quest2}
What is the asymptotic behaviour of $\mathrm{V}^{\l}(x)$?
\end{ques}

\begin{ques}\label{quest1}
What multiplicities of elements in $\mathscr{V}^{\l}$ are possible?
\end{ques}

We did not find in the literature papers addressed to answer the
above two questions. In regarding to the order of the set $\mathrm{V}(x)$ 
of all totients not greater than $x$, K. Ford obtained its exact order, 
see \cite[Thm. 1]{F1998}. Before Ford, H. Maier and C. Pomerance, cf. \cite{MaPom}, 
obtained a nice order for $\mathrm{V}(x)$ that will be useful here. In \cite{Pillai} S. Pillai proved 
that the set of multiplicity of totients is unbounded, he proved that 
$$\limsup_{x\rightarrow\infty}\{\mathrm{A}(m)\,;\,m\in\mathscr{V}(x)\}=\infty.$$
Pillai's theorem can be considered as an insight to Sierpi\'nski's 
conjecture that says that for each $k>1$ there is $m\in\mathbb{N}$ such 
that $\mathrm{A}(m)=k$. Sierpi\'nski's conjecture was completely solved by 
K. Ford in \cite{F1999}.
%
%
%


We divide this short paper in 3 sections. The section 2 deals with the 
simplest (and very easy) case where $\l=1$. The multiplicity of totients in $
\mathscr{V}^{1}$ are just $2$ or $4$, and the elements in the pre-images
of such totients are just a power of an odd prime and twice this power of an odd prime. 
So the main idea is that totients in $\mathscr{V}^{1}(x)$ are just image 
of prime numbers $\equiv 3\mod 4$, except those too rare totients in $
\mathscr{V}^{1}$ whose pre-image has a power of prime number with 
exponent bigger than one. Hence $\mathscr{V}^{1}(x)$ has asymptotic 
order $\pi(x)/2$, where $\pi(x)$ stands for the number of primes numbers 
not bigger than $x$, see Corollary \ref{main}.

In section 3 we deal with the case 
$\l\geq 2$. While the multiplicities of totients in $\mathscr{V}^{1}$ are 
far from being exciting, the multiplicities of totients in $\mathscr{V}^{\l}$ 
with $\l\geq 2$ seems to be unbounded, i.e. $\limsup_{m\in\mathscr{V}^{\l}}
\mathrm{A}(m)=\infty$.
It is very simple to see that if 
there is some $\l_0$ such that the set of multiplicities of totients in 
$\mathscr{V}^{\l_0}$ is unbounded, then the set of multiplicities of totients in 
$\mathscr{V}^{\l_0+n}$ is also unbounded for all $n\geq 0$, c.f. Proposition 
\ref{cond1}. We are able to show that Dickson's $k$-tuples Conjecture implies that $\l_0=2$,
see Theorem \ref{Dick2}. The remaining of Section 3 is devoted to provide a different 
approach to get information on $\l_0$.
The main idea lies 
in Theorem \ref{t.2} and involves of finding a 
suitable lower bound for the number $\mathrm{S}^{\l}(x)$ of the solutions in 
$\mathscr{V}^{\l}(x)$ of the equation $\phi(z)\leq x$. To do so, first we provide the following
upper bound addressing to
Question \ref{quest2} above
$$\mathrm{V}^{\l}(x)=\Oh_{\l}\left(\dfrac{x}{\log x}(\log\log x)^{\l}\right),$$ that
derives from a classical result due to G. Hardy and S. 
Ramanujan, see Theorem \ref{OhVl}. We also show the inequality 
$$ S^{\ell+n}(2^n x)\ge \dfrac{S^{\ell}(x)}{2^n}\ \forall\,n\geq 0, $$
c.f. Lemma \ref{lemal0n}.  Now, for each real number $x>2$, there is a positive integer
$\tilde{\l}(x)$ such that $\mathrm{S}^{\tilde{\l}(x)}(x)>\frac{1}{2^{\tilde{\l}(x)}}\mathrm{V}(x)$.
So the key is to show the existence of a positive integer $\l_0$ that is given by
$$\l_0:=\liminf_{x\rightarrow\infty}\min\{\tilde{\l}(x)\}.$$
If we assume that such $\l_0$ exists, through the equation
\begin{equation*}
\mathrm{S}^{\l_0+n}(2^nx_i)\geq \dfrac{1}{2^n}\mathrm{S}^{\l_0}(x_i)>\dfrac{1}{2^{\l_0+n}}\dfrac{\mathrm{V}(x_i)}{\mathrm{V}^{\l_0+m}(2^{n}x_i)}\mathrm{V}^{\l_0+m}(2^{n}x_i),
\end{equation*}where $(x_i)$ is a  suitable sequence of positive real numbers that goes to infinity,
and using a result involving the above order of $\mathrm{V}^{\l}(x)$ and the order of $\mathrm{V}(x)$ due to  H. Maier and C.
Pomerance, c.f. Corollary \ref{mc1}, then we can show
$$\limsup_{m\in\mathscr{V}^{\l_0+n}}\mathrm{A}(m)=\infty, \ \forall n\geq 0,$$
that is precisely the contents of our Theorem \ref{main2}.

The proof of the existence of the above $\liminf$ seems to be workable, see some computations on Table ~\ref{tabela2} below. 
We strongly believe that $\l_0$ is equal to $2$, as suggested by Dickson's $k$-tuples conjecture.
We also note that proving that $\l_0$ is a positive integer, then for each $\l\geq \l_0$,
the multiplicities of $\mathscr{V}^{\l}$ are unbounded. So one should determine all the positive integers that are realized
as the multiplicity of totients in $\mathscr{V}^{\l}$, equivalently, the number of the solutions of the equation $\phi(x)=m$ where $m\in\mathscr{V}^{\l}$.
Hence, to show that $\l_0$ exists, is an insight to a Sierpi\'nski-type problem on each stratum $\mathscr{V}^{\l}$. 

\bigskip

\textbf{Acknowledgements:} We would like to express our appreciation to Kevin Ford for his valuable and constructive suggestions during the development of this paper.
 
\section{The simplest case $\l=1$}

The set of prime numbers is denoted by $\nP$ and $\pi(x)$ stands for the 
number of prime numbers not greater than $x$. We also use the 
\textit{big} $\Oh$  and \textit{small} $\oh$ standard notations. The 
following notation will be useful throughout this paper: given any subset 
$\mathrm{U}$ of the positive integers and $x\in\mathbb{R}$ a real number, 
$\mathrm{U}(x)$ stands for the elements of $\mathrm{U}$
not greater than $x$, $$\mathrm{U}(x):=\{n\in\mathrm{U};\,n\leq x\}\,,$$ and it is 
clear that $\vert \mathrm{U}(x)\vert$ denotes the magnitude of $\mathrm{U}(x)$.

\begin{lem}\label{lemmaA}
Given an integer number $t>0$ and considering the set 
$$\mathcal{R}_{t}:=\{k\in\mathscr{V}\,\vert\,q^j\in \phi^{-1}(k)\mbox{ with } j\geq t\mbox{ and } q\in\mathcal{P}\},$$ we get
$$\vert\mathcal{R}_{t}(x)\vert=\oh(\sqrt[t]{x})\,.$$
\end{lem}

\begin{proof} 
For every $k\in \mathcal{R}_{t}(x)$ there is a prime number $q$ and an integer $m\geq t$ such that 
	$$x \geq k=q^m-q^{m-1}\ge q^{t}-q^{t-1}\geq q^{t}/2.$$
Hence, from the above inequality we get the upper bound $q\leq \sqrt[t]{2x}$. 
Thus, by the Prime Number Theorem
\begin{equation*}
\dfrac{\vert\mathcal{R}_{t}(x)\vert}{\sqrt[t]{2x}}\leq\dfrac{\pi(\sqrt[t]{2x})}{\sqrt[t]{2x}},
\end{equation*} and so $\mathcal{R}_{t}(x)=\oh(\sqrt[t]{2x})=\oh(\sqrt[t]{x})$.
\end{proof}

For the sake of organization of the presentation we just state the next result, omitting its quite trivial proof.

\begin{lem}\label{lemmaB}
For each odd positive integer, $\A(2r)\in\{0,2,4\}$.
If $\A(2r)=2$, then $\phi^{-1}(2r)=\{p^{n}, 2p^{n}\}$, with $p$ an odd 
prime and $n>0$. If $\A(2r)=4$, then $2r+1$ is a prime number and 
$\phi^{-1}(2r)=\{2r+1, q^{m}, 4r+2, 2q^{m}\}$ with $q$ a prime number and $m>1$.
\end{lem}

%

Next we study the distribution of totients $2$ modulo $4$. 
Let us start by taking the following useful sets
$$
\mathscr{V}^{1}_{k}=\{2r\,;\, r  \ \mbox{odd and} \ \A(2r)=k\},\mbox{ for each } k=2,4.
$$

\begin{center}
	\begin{table}[H]
		\caption{The number of totients  $2$ mod $4$ $\leq x$ with a fixed multiplicity}
		\begin{tabular}{ccccc}
			$x$ & $\pi(x)$ & $\vert\mathscr{V}_{2}^{1}(x)\vert$ & $\vert\mathscr{V}^{1}_{4}(x)\vert$ & $\vert\mathscr{V}_{2}^{1}(x)\vert/\pi(x)$ \\ \hline
			$10^3+2$ & $168$ & $87$ & $5$ & $0.517857\dots$ \\ \hline
			$10^4+2$ & $1229$ & $625$ & $8$ & $0.508543\dots$ \\ \hline
			$10^5+2$ & $9592$ & $4831$ & $14$ & $0.503648\dots$ \\ \hline
			$10^6+2$ & $78498$ & $39400$ & $20$  & $0.501923\dots$ \\ \hline 
			$10^7+2$ & $664579$ & $332606$ & $34$ & $0.500476\dots$  \\ \hline
			$10^8+2$ & $5761455$ & $2881495$ & $78$  & $0.500133\dots$ \\
		\end{tabular}
	\end{table}
\end{center}

\begin{cor}\label{A}
$\lim_{x\to \infty}\dfrac{\vert\mathscr{V}^{1}_{4}(x)\vert}{\sqrt{x}}=0.$
\end{cor}
\begin{proof}
Just note that $\mathscr{V}^{1}_{4}(x)\subseteq \mathcal{R}_{2}(x)$.
\end{proof}

Let us introduce some useful notation, $\mathcal{P}(k,j)=\{p\in \mathcal{P}\,;\ p\equiv j \, \mod k\}$ and $\pi(x;k,j)=|\mathcal{P}(x;k,j)|$. 
It is just easy to see that if $p\in \phi^{-1}(2r)$ with $r\equiv 1\mod 2$, then $p\in \mathcal{P}(4,3).$

\begin{thm}\label{thmA} For every real number $x>0$,
$$\vert\mathscr{V}^{1}_2(x)\vert \sim \frac{\pi(x)}{2}$$
\end{thm}
\begin{proof}
Let us consider the subset $\mathcal{A}$ of $\mathscr{V}^{1}$ given by the elements $m\equiv 2\mod 4$ 
such that $m+1$ is prime, and $\mathcal{B}:=\mathscr{V}^{1}\setminus\mathcal{A}$.
Thus
$$\vert \mathscr{V}^{1}_2(x)\vert=\vert \mathcal{A}(x)\vert+ \vert\mathcal{B}(x) \vert.$$
Moreover, from the above two Lemmas, $\vert\mathcal{B}(x) \vert=\oh(\sqrt{x})$. Since $\sqrt{x}=\oh(\pi(x))$, 
the Prime Number Theorem implies that $\vert\mathcal{B}(x) \vert=\oh(\pi(x))$. Hence
\begin{equation}
  \liminf_{x\to \infty}\frac{\vert \mathscr{V}^{1}_2(x)\vert}{\pi(x)}=\liminf_{x\to \infty}\frac{\vert \mathcal{A}(x)\vert}{\pi(x)} \ \mbox{and} 
  \ \limsup_{x\to \infty}\frac{\vert \mathscr{V}^{1}_2(x)\vert}{\pi(x)}=\limsup_{x\to \infty}\frac{\vert \mathcal{A}(x)\vert}{\pi(x)}.
\end{equation}  
We claim that $$\vert\mathcal{A}(x)\vert\sim \frac{\pi(x)}{2}.$$
Since $m+1$ must be a prime equivalent to $3$ module $4$, we get
$$\{m+1\in \mathcal{P}(x+1;4,3)\}=\mathcal{A}(x)\cup \{m+1\in \mathcal{P}(x+1;4,3); \A(m)=4\}.$$
Thus
$$\pi(x;4,3)=\vert\mathcal{A}(x)\vert+\vert\{m+1\in \mathcal{P}(x+1;4,3); \A(m)=4\}\vert.$$
We know that $\vert\{m+1\in \mathcal{P}(x+1;4,3); \A(m)=4\}\vert=\oh(\pi(x))$.
Now, the Prime Number Theorem in Arithmetic Progression assure that
$$\pi(x;4,3)\sim \frac{\pi(x)}{2}.$$
Hence
$$\lim_{x\to \infty}\dfrac{\vert\mathrm{A}(x)\vert}{\pi(x)}=\dfrac{1}{2}.$$
This finishes the claim and proves the theorem.
\end{proof}

From the above Theorem and the Corollary \ref{A} we have
$$\lim_{x\to \infty}\dfrac{\mathrm{V}^{1}(x)}{\pi(x)}=\lim_{x\to \infty}\dfrac{\vert\mathscr{V}^{1}_2(x)\vert+\vert\mathscr{V}^{1}_4(x)\vert}{\pi(x)}=\lim_{x\to \infty}\dfrac{\vert\mathscr{V}^{1}_2(x)\vert}{\pi(x)},$$
and so get the main result of this section.
\begin{cor}\label{main}
$\vert\mathscr{V}^{1}(x)\vert\sim \dfrac{\pi(x)}{2}.$
\end{cor}

\begin{rem}
It follows from the above Theorem that for each $k>1$
$$\vert\{m\in\mathscr{V}^{k}(x)\,;\, \A(m)\ge 2\}\vert\gg \dfrac{\pi(x)}{2^k}.$$ 
\end{rem}

\begin{cor}\label{corA}
$\vert\mathscr{V}^{1}_4(x)\vert=\oh(\vert\mathscr{V}^{1}_2(x) \vert).$
\end{cor}
\begin{proof}
Since $\vert\mathscr{V}^{1}_2(x)\vert \sim \dfrac{x}{2\log x}$ and $\vert\mathscr{V}^{1}_4(x)\vert=\oh(\sqrt{x})$, we get
$$\frac{\vert \mathscr{V}^{1}_4(x)\vert}{\vert\mathscr{V}^{1}_2(x)\vert}=\Oh\left(\frac{\sqrt{x}\log x}{x}\right)=\Oh\left(\frac{\log x}{\sqrt{x}}\right).$$
Therefore,
$$\lim_{x\to \infty}\frac{\vert\mathscr{V}^{1}_4(x)\vert}{\vert\mathscr{V}^{1}_2(x)\vert}=0.$$
\end{proof}

\section{When $\l\geq 2$}

The next natural step is try to answer the Questions \ref{quest2} and 
\ref{quest1} when $\l=2$, i.e. considering totients 
that are $4\mod 8$. So are allowed until two odd prime divisors
in the pre-image of such totients, hence the 
computations become more involved. The first examples suggest that the 
multiplicities of totient in $\mathscr{V}^{2}$ can assume very large values. For 
example, using a simple computer one can see in a few minutes that the 
multiplicities of totients in $\mathscr{V}^{2}(10^7)$ assume all values 
between $2$ and $35$ and the biggest one is $42$, there is some gaps between
$35$ and $42$ for such $x$. In a couple of hours one 
can see that all numbers between $2$ and $72$ are realized as 
multiplicities of totients in $\mathscr{V}^{2}(3\cdot 10^9+2\cdot10^6)$, 
and the maximal multiplicity attained is $94$. Doing the same 
computations for $\l=3,4,5,6,7$ very 
large multiplicities also appear, but as $\l$ increases, large multiplicities are 
attained by even smaller totients. For example, all numbers between 
$2$ and $82$ are realized as the multiplicities of 
totients in $\mathscr{V}^{3}(10^6)$ and the biggest multiplicity is $169$, see Table ~\ref{tabelamax}.
\begin{center}
	\begin{table}[H]
		\caption{maximal multiplicity in $\mathscr{V}^{\l}(x)$}
		\begin{tabular}{cccccc}
			 & $x=10^6$ & $x=5\cdot10^6$ & $x=10^7$ & $x=5\cdot10^7$  \\ \hline
			$\max\mathrm{A}(\mathscr{V}^{2}(x))$ & 32 & 34 & 42 & 57 \\ \hline
			$\max\mathrm{A}(\mathscr{V}^{3}(x))$ & 169 & 250 & 277 & 427 \\ \hline
			$\max\mathrm{A}(\mathscr{V}^{4}(x))$ & 463 & 745 & 860 & 1427 \\ \hline
			$\max\mathrm{A}(\mathscr{V}^{5}(x))$ & 998 & 1804 & 1961 & 3732 \\ \hline
			$\max\mathrm{A}(\mathscr{V}^{6}(x))$ & 1401 & 2222 & 3887 & 6239 \\ \hline
			$\max\mathrm{A}(\mathscr{V}^{7}(x))$ & 1375 & 3258 & 4076 & 7807 \\
		\end{tabular}
		\label{tabelamax}
	\end{table}
\end{center}
Now, it is only natural to make a week version of Question \ref{quest1} in the following way.

\begin{ques}\label{q1}
Is $\limsup_{m\in\mathscr{V}^{\l}}\mathrm{A}(m)=\infty$ for each $\l\geq 2$?
\end{ques}

Next we see an elementary result concerning the above Question \ref{q1}.

\begin{prop}\label{cond1} If there exists $\l_0$ such that $\limsup_{m\in\mathscr{V}^{\l_0}}\mathrm{A}(m)=\infty$, then for each $n\geq 0$,
$\limsup_{m\in\mathscr{V}^{\l_0+n}}\mathrm{A}(m)=\infty$.
\end{prop}
\begin{proof}
Take $m\in\mathscr{V}^{\l_0}$ such that $\mathrm{A}(m)=k$.
For each prime number $p\equiv 3 \mod 4$ not dividing any
element of $\phi^{-1}(m)$ it follows $\mathrm{A}((p-1)m)\geq k$ and we are done.
\end{proof}

The above Proposition shows that we should prove the existence of a suitable 
$\l_0$, ideally the smallest one, where $\limsup_{m\in\mathscr{V}^{\l_0}}\mathrm{A}(m)=\infty$.
Now we give a conditional proof that such $\l_0=2$ using Dickson's 
$k$-tuples Conjecture, c.f. \cite{Dick}.

\begin{thm}\label{Dick2}
Dickson's  $k$-tuples Conjecture implies that $\l_0=2$.
\end{thm}
\begin{proof} 
%
%
%
%
%
%
%

We start by taking a positive integer $k$ and consider the following two sequences of linear forms
\begin{eqnarray*}
f_i(n):=1+2\cdot 3^i\cdot5^{k+1-i}(2n+1) & \mbox{and} &  g_i(m):=1+2\cdot 3^{k+1-i}\cdot 5^{i}(2m+1)
\end{eqnarray*} where $i=1,\dots, k$. Now let us take a prime number $p$. If we assume that $p$ divides both $f(p)$ and $f(p+1)$ 
(respectively, $g(p)$ and $g(p+1)$) then $p$ must be $2$, $3$ or $5$ and so $p$ divides $1$, that is a contradiction.
So the above two sequences are admissible, c.f. \cite{Rib}. Hence
Dickson's $k$-tuples Conjecture implies that there are infinitely many positive integers $n$ and $m$ such that
$p_1:=f_1(n),\dots, p_k:=f_k(n)$ and $q_1:=g_1(m),\dots,q_k:=g_k(m)$ are all prime numbers.
Thus for all $1\leq i,j\leq k$ we have
\begin{equation*}
m:=\phi(p_i\cdot q_i)=(p_i-1)(q_i-1)=(p_j-1)(q_j-1)=\phi(p_j\cdot q_j)
\end{equation*} and so $\mathrm{A}(m)\geq k$ with $m\in\mathcal{V}^2$.
%
%
\end{proof}

The remaining of this paper is devoted to give a sufficient condition on the 
existence of a suitable $\l_0$ using a completely different approach.
We start by considering a family of suitable
functions $\mathrm{S}^{k}$, with $k\in\N$, such that for each 
$x\in\mathbb{R}$ its value by $\mathrm{S}^{k}$ is the sum
of all multiplicities of the totients in $\mathscr{V}^{k}(x)$, namely
$\mathrm{S}^{k}(x):=\sum_{m\in\mathscr{V}^{k}(x)}\mathrm{A}(m)$.

%
%
%
%



\begin{thm}\label{t.2}
If there is a function $f:\mathbb{R}\rightarrow \mathbb{R}$ such that
\begin{enumerate}
\item[i.] $\liminf_{x\to \infty}f(x)=\infty$, and
\item[ii.] $\mathrm{S}^{\l}(x_i)\geq V^{\l}(x_i) \cdot f(x_i)$, for some 
increasing sequence $(x_i)_{i\in \mathbb{N}}$,
\end{enumerate} then $\limsup_{m\in\mathscr{V}^{\l}}\A(m)=\infty.$
\end{thm}
\begin{proof}
We start by noting that
	$$\mathrm{S}^{\l}(x)=\sum_{m\in\mathscr{V}^{\l}(x)}\mathrm{A}(m)\leq V^{\l}(x)\cdot \max\{\A(m); m\in \mathscr{V}^{\l}(x)\}.$$
If we assume that $\A(m)=\oh(f(m))$ with $m\in \mathscr{V}^{\l}(x)$, then 
\begin{equation}\label{ozinho1}
\dfrac{\mathrm{S}^{\l}(x)}{f(x)\cdot V^{\l}(x)}=\oh(1).
\end{equation}Now, condition (ii) ensures that
	$$\liminf_{x\to \infty}\dfrac{\mathrm{S}^{\l}(x)}{f(x)\cdot V^{\l}(x)}\geq 1,$$
which contradicts equation \eqref{ozinho1}. Hence $\A(n)\neq \oh(f(n))$, and
from condition (1) we are done. 
\end{proof}

Of course that we seek for some function $f$ satisfying the 
hypothesis of the above Theorem \ref{t.2}. The first step is try
to  measure the magnitude of $\mathrm{V}^{\l}(x)$, 
addressing to Question \ref{quest2}. This is the content of
the following theorem.

\begin{thm}\label{OhVl}
For each $\l\geq 1$, $\mathrm{V}^{\l}(x)=\Oh_{\l}\left(\dfrac{x}{\log x}(\log\log x)^{\l}\right).$
\end{thm}

\begin{proof}
Since the number of prime divisors of an element in the pre-image of a totient in $\mathscr{V}^{\l}$ is at most $\l+1$, it follows that 
$$\mathrm{V}^{\l}(x)<\sum_{i=1}^{\l+1}\pi_{i}(x),$$ where
$\pi_{i}(x)=\#\{n\leq x\,;\,\omega(n)=i\}$ is the number of
elements not bigger than $x$ such that in their prime factorization appear exactly $i$ 
different prime numbers. In \cite{HR} Hardy and Ramanujan proved that 
$$\pi_{i}(x)<M\left(\dfrac{x}{\log x}\dfrac{(\log\log x+c)^{i-1}}{(i-1)!}\right).$$
From the above above two inequalities the results follows easily as follows.
\begin{eqnarray*}
\mathrm{V}^{\l}(x)=\Oh_{\l}\left(\sum_{i=1}^{\l+1}\dfrac{x}{\log x}\dfrac{(\log\log x)^{\l}}{(i-1)!}\right)=\Oh_{\l}\left(\dfrac{x}{\log x}(\log\log x)^{\l}\right).
\end{eqnarray*}
\end{proof}

We strong believe that the above upper bound for 
$\mathrm{V}^{\l}(x)$ can be improved 
using better sieve methods. Compare, for example, the above 
bound in the particular case $\l=1$ to Corollary \ref{main} of 
the previous section. On the other hand, for the purposes of 
this paper the above upper bound fits nicely. The next theorem is due to Maier and 
Pomerance and it can be found in \cite{MaPom}.

\begin{thm}[Maier--Pomerance]\label{o.2}
$$V(x) = \dfrac{x}{\log x} \exp\left((C + \oh(1))(\log \log \log x)^2\right).$$
\end{thm}

\noindent The proof of the following useful result follows easily from Theorem \ref{OhVl} and from Theorem \ref{o.2}.

\begin{cor}\label{mc1}
For any $\ell\in \mathbb{N}$ and any real number $M>0$ we have $$\lim_{x\to \infty}\dfrac{V(x)}{V^{\ell}(M x)}=\infty.$$
\end{cor}


\noindent For each real number $x>0$, let $k_0:=k_0(x)$ be the smallest natural 
number such that $\mathrm{S}^{\l}(x)=0$ for every $\l>k_0$. 

\begin{lem}
 $k_0(x)=\left\lfloor \dfrac{\log x}{\log 2}\right \rfloor$.
 \end{lem}
 \begin{proof}Let $x>0$ be a real number. Since $\phi(2^{\l+1})=2^{\l}$ is the
smallest totient in $\mathscr{V}^{\l}(x)$, $\mathrm{S}^{\l}(x)\neq 0$ if and only if $x\geq 2^{\l}$.
 Hence $k_0(x)=\left\lfloor \dfrac{\log x}{\log 2}\right \rfloor$.
\end{proof}

\begin{lem}\label{lemal0n}
For every $x>1$, $S^{\ell+n}(2^n x)\ge \dfrac{S^{\ell}(x)}{2^n}$, $\forall\,n\geq 1$.
\end{lem}
\begin{proof}
Set $t=S^{\ell}(x)$ and take $n_1,n_2,...,n_t$ with $n_i\neq n_j$ such that $\phi(n_i)\in \mathcal{V}^{\ell}(x)$. 
Consider $A=\{i;\, n_i \ \mbox{is odd}\}$ and $B=\{i;\,n_i \ \mbox{is even}\}.$ Obviously $t=|A\cup B|$. 
We easily see that $|A|\le |B|$ and then  $t\le 2\cdot |B|$. Now, for each $i\in B$,
$$\phi(2n_i)=2\phi(n_i)\in \mathcal{V}^{\ell+1}(2x)$$ 
and so $\mathrm{S}^{\l+1}(2x)\geq \mathrm{S}^{\l}(x)\cdot 2^{-1}$.
\end{proof}


\noindent Given a real number $x$, there is $2\leq \tilde{\l}(x)\leq k_0(x)$ such that $\mathrm{S}^{\tilde{\l}(x)}(x)>\dfrac{1}{2^{\tilde{\l}(x)}}\mathrm{V}(x)$.
Otherwise $x<\mathrm{S}(x)=\sum_{i=1}^{k_{0}(x)}\mathrm{S}^{i}(x)<\mathrm{V}(x)$, getting a contradiction.
Hence, for each positive real number $x$, set
$$\l(x):=\min\{\tilde{\l}(x)\}.$$

\begin{hypo}\label{hypo}
There is $\l_0:=\displaystyle\liminf_{x\rightarrow\infty}\l(x)<\infty$.
\end{hypo}

We tried to give a proof of the Hypothesis, but unfortunately we were not able until now. Computations
suggest that $\l_0$ does exist, see Table ~\ref{tabela2}. It is also clear that Dickson's $k$-tuples Conjecture
is stronger that our Hypothesis.
\begin{center}
	\begin{table}[H]
		\caption{Collecting $(2^{\l}\cdot\mathrm{S}^{\l}(x))/\mathrm{V}(x)$}
		\begin{tabular}{lccccc}
			 & $x=10^6$ & $x=5\cdot10^6$ & $x=10^7$ & $x=5\cdot10^7$ & $x=10^8$ \\ \hline	
		$\l=2$ & $4,6044\ldots$ & $4,4593\ldots$ & $4,4033\ldots$ & $4,2755\ldots$ & $4,2223\ldots$ \\ \hline
		$\l=3$	& $13,3598\ldots$ & $13,3172\ldots$ & $13,2956\ldots$ & $13,2309\ldots$ & $13,1984\ldots$ \\ \hline
		$\l=4$	& $29,8629\ldots$ & $30,5975\ldots$ & $30,8712\ldots$ & $31,4799\ldots$ & $31,7215\ldots$ \\ \hline
		$\l=5$	& $54,3445\ldots$ & $57,7380\ldots$ & $59,0381\ldots$ & $61,9642\ldots$ & $63,1584\ldots$ \\ \hline
		$\l=6$	& $86,5368\ldots$ & $95,3994\ldots$ & $98,8584\ldots$ & $107,0031\ldots$ & $110,4274\ldots$ \\ \hline
		$\l=7$	& $125,1485\ldots$ & $143,0769\ldots$ & $150,4968\ldots$ & $167,9442\ldots$ & $175,5206\ldots$ \\ \hline
		$\l=8$	& $169,4372\ldots$ & $200,5425\ldots$ & $213,6786\ldots$ & $245,7202\ldots$ & $260,0693\ldots$ \\ \hline
		$\l=9$	& $215,2255\ldots$ & $265,4653\ldots$ & $285,8896\ldots$ & $339,0014\ldots$ & $363,3946\ldots$\\ \hline
		$\l=10$	& $264,8755\ldots$ & $333,3547\ldots$ & $364,4399\ldots$ & $445,1494\ldots$ & $482,6826\ldots$ \\ \hline	
		$\l=11$	& $289,2330\ldots$ & $392,9655\ldots$ & $441,5500\ldots$ & $557,9906\ldots$ & $613,1423\ldots$ \\ \hline			  
	    $\l=12$	& $326,8434\ldots$ & $452,8810\ldots$ & $504,6976\ldots$ & $666,9787\ldots$ & $748,3791\ldots$ \\ \hline	
		$\l=13$	& $325,5704\ldots$ & $488,7035\ldots$ & $567,5796\ldots$ & $768,3839\ldots$ & $869,4729\ldots$\\ \hline	
		$\l=14$	& $290,2446\ldots$ & $497,8492\ldots$ & $598,7011\ldots$ & $855,6433\ldots$ & $984,6031\ldots$ \\ \hline	
		$\l=15$	& $299,1556\ldots$ & $492,8571\ldots$ & $598,0896\ldots$ & $908,2700\ldots$ & $1076,7877\ldots$\\ 	
		\end{tabular}
		\label{tabela2}
	\end{table}
\end{center}

\begin{thm}\label{main2}
Assuming the Hypothesis, $$\limsup_{m\in\mathscr{V}^{\l_0+n}}\mathrm{A}(m)=\infty, \ \forall n\geq 0.$$
\end{thm}
\begin{proof}We first prove in the case that $n=0$. 
Take $\l_0$ as in Hypothesis. Thus there is a sequence $(x_i)$ with $x_i\rightarrow\infty$ such that
\begin{equation*}
\mathrm{S}^{\l_0}(x_i)>\dfrac{1}{2^{\l_{0}}}\mathrm{V}(x_i)=\dfrac{1}{2^{\l_{0}}}\dfrac{\mathrm{V}(x_i)}{\mathrm{V}^{\l_0}(x_i)}\mathrm{V}^{\l_0}(x_i).
\end{equation*}
From Corollary \ref{mc1}, we know that $2^{-\l_0}\frac{\mathrm{V}(x_i)}{\mathrm{V}^{\l_0}(x_i)}$ 
goes to infinity when $i\rightarrow\infty$. Hence the Theorem \ref{t.2} assures that $\limsup_{m\in\mathcal{V}^{\l_0}}\,\mathrm{A}(m)=\infty$.

Now take any integer $n\geq 1$. From Lemma \ref{lemal0n} and the Hypothesis there is a 
sequence $(x_i)$ that goes to infinity such that
\begin{equation*}
\mathrm{S}^{\l_0+n}(2^nx_i)\geq \dfrac{1}{2^n}\mathrm{S}^{\l_0}(x_i)>\dfrac{1}{2^{\l_0+n}}\dfrac{\mathrm{V}(x_i)}{\mathrm{V}^{\l_0+m}(2^{n}x_i)}\mathrm{V}^{\l_0+m}(2^{n}x_i)
\end{equation*} and again by the Corollary \ref{mc1} and Theorem \ref{t.2} 
we are done.
\end{proof}

\end{document}